\documentclass[12pt]{amsart}
\usepackage{amsfonts}
\usepackage{amsmath}
\usepackage{amssymb,latexsym}
\usepackage[mathcal]{eucal}
\usepackage{amscd}
\input xy
\xyoption{all}

\newcommand{\Acal}{\mathcal{A}}
\newcommand{\Bcal}{\mathcal{B}}

\newcommand{\Fcal}{\mathcal{F}}

\newcommand{\Xcal}{\mathcal{X}}
\newcommand{\Ycal}{\mathcal{Y}}
\newcommand{\Zcal}{\mathcal{Z}}

\newcommand{\ch}{\mathbf{1}}

\newcommand{\Z}{\mathbb{Z}}

\newcommand{\R}{\mathbb{R}}
\newcommand{\C}{\mathbb{C}}
\newcommand{\N}{\mathbb{N}}

\newcommand{\PP}{\mathbb{P}}

\newcommand{\al}{\alpha}
\newcommand{\Ga}{\Gamma}

\newcommand{\del}{\delta}

\newcommand{\ep}{\epsilon}
\newcommand{\sig}{\sigma}

\newcommand{\la}{\lambda}

\newcommand{\tet}{\theta}

\newcommand{\om}{\omega}
\newcommand{\Om}{\Omega}

\newcommand{\br}{\vspace{3 mm}}

\newcommand{\tri}{\bigtriangleup}

\newcommand{\cls}{{\rm{cls\,}}}

\newcommand{\id}{{\rm{id}}}

\newcommand{\supp}{{\rm{supp\,}}}

\newcommand{\ess}{{\rm{ess}}}

\newcommand{\prox}{{\rm{prox}}}
\newcommand{\mpr}{{\rm{mp}}}

\newcommand{\Homeo}{\rm{Homeo\,}}

\swapnumbers \theoremstyle{plain}
\newtheorem{thm}{Theorem}[section]
\newtheorem{cor}[thm]{Corollary}
\newtheorem{lem}[thm]{Lemma}
\newtheorem{prop}[thm]{Proposition}

\theoremstyle{definition}
\newtheorem{defn}[thm]{Definition}
\newtheorem{defns}[thm]{Definitions}

\newtheorem{rmk}[thm]{Remark}

\newtheorem{prob}[thm]{Problem}

\numberwithin{equation}{section}

\begin{document}


\title
{Stationary dynamical systems}
\footnote{A preliminary version of this work has been in circulation
as a preprint for several years now but for technical reasons was not
previously submitted for publication.}
\author{Hillel Furstenberg }

\address{Department of Mathematics\\
     Hebrew University of Jerusalem\\
         Jerusalem\\
         Israel}
\email{harry@math.huji.ac.il}

\author{Eli Glasner }

\address{Department of Mathematics\\
     Tel Aviv University\\
         Ramat Aviv\\
         Israel}
\email{glasner@math.tau.ac.il}

\date {September 16, 2009}

\keywords{Stationary systems, $m$-systems, $SL(2,\R)$, Szemer\'edi,
stiff, WAP, SAT}



\thanks{{\it 2000 Mathematics Subject Classification.}
Primary 22D40, 22D05, 37A50 Secondary 37A30, 37A40}

\begin{abstract}
Following works of Furstenberg and Nevo and Zimmer 
we present an outline of a theory of stationary 
(or $m$-stationary) dynamical systems for a general acting group
$G$ equipped with a probability measure $m$. Our purpose is two-fold:
First to suggest a more abstract line of development, including
a simple structure theory. Second, to point out some interesting
applications; one of these is a Szemer\'edi type theorem for $SL(2,\R)$.
\end{abstract}

\maketitle


\tableofcontents \setcounter{secnumdepth}{1}



\setcounter{section}{0} 
\setcounter{page}{1}


\section*{Introduction}

Classical ergodic theory was developed for the group
of real numbers $\R$ and the group of integers $\Z$. 
Later generalizations to $\R^d$ and $\Z^d$ actions
evolved and more recently the theory has been vastly extended 
to handle more general concrete and abstract amenable groups.
There however the theory finds a natural boundary,
since by definition it deals with measure preserving actions 
on measurable or compact spaces, and these need not
exist for a non-amenable group.
Of course semi-simple Lie groups or non-commutative free groups  
admit many interesting measure preserving actions, but for many other natural
actions of these groups no invariant measure exists.

Following works of Furstenberg (e.g. \cite{F}, \cite{F1}, \cite{F73}, \cite{F80}) 
and Nevo and Zimmer (e.g. \cite{NZ1}, \cite{NZ2}, \cite{NZ3}), 
we present here an outline of a theory of stationary 
(or $m$-stationary) dynamical systems for a general acting group
$G$ equipped with a probability measure $m$.
By definition such a system comprises a compact metric space $X$
on which $G$ acts by homeomorphisms and a probability measure
$\mu$ on $X$ which is $m$ stationary; i.e. it satisfies the 
convolution equation $m*\mu = \mu$. The immediate advantage
of stationary systems over measure preserving ones is the fact that,
given a compact $G$-space $X$, an $m$-stationary measure always exists
and often it is also quasi-invariant.

The aforementioned works, 
as well as e.g. \cite{K} and the more recent works
\cite{BS} and \cite{BQ}, amply demonstrate the potential
of this new kind of theory and our purpose here is two-fold.
First to suggest a more abstract line of development, including
a simple structure theory, and second, to point out some interesting
applications. 

We thank Benjy Weiss for substantial contributions to this work. These
were communicated to us via many helpful discussions
during the period in which this work was carried out.

\section{Stationary dynamical systems}

{\bf Definitions:} Let $G$ be a locally compact second countable
topological group, $m$ an {\bf admissible} probability measure on $G$. 
I.e. with the following two properties: 
(i) For some $k \ge 1$ the convolution power 
$\mu^{*k}$ is absolutely continuous with respect to Haar measure.
(ii) the smallest closed subgroup containing $\supp(m)$
is all of $G$. 
Let $(X,\Bcal)$ be a standard Borel space and let
$G$ act on it in a measurable way. A probability measure $\mu$ on
$X$ is called {\bf $m$-stationary}, or just {\bf stationary} when $m$ is
understood, if $m*\mu=\mu$. 
As shown by Nevo and Zimmer, every $m$-stationary
probability measure $\mu$ on a $G$-space $X$ is quasi-invariant;
i.e. for every $g\in G$, $\mu$ and $g\mu$ have the same null sets.

Given a stationary measure $\mu$ the
quintuple $\Xcal =(X,\Bcal,G,m,\mu)$ is called an {\bf
$m$-dynamical system}, or just an $m$-{\bf system\/}. (Usually we omit the
$\sig$-algebra $\Bcal$ from the notation of an $m$-system, and
often also the group $G$ and the measure $m$). An $m$-system
$\Xcal $ is called {\bf measure preserving} if the stationary
measure is in fact $G$-invariant. With no loss of generality we
may assume that the Borel space $X$ is a compact metric space and
that the action of $G$ on $X$ is by homeomorphisms. For a compact
metric space $X$, the space of probability Borel measures on $X$
with the weak* topology will be denoted by $M(X)$; it is a compact
convex metric space. When $G$ acts on $X$ by homeomorphisms the
closed convex subset of $M(X)$ consisting of $m$-stationary
measures will be denoted by $M_m(X)$. By the Markov-Kakutani fixed
point theorem $M_m(X)$ is non-empty. We say that the $m$-system
$(X,\mu)$ is {\bf ergodic} if $\mu$ is an extreme point of
$M_m(X)$ and that it is {\bf uniquely ergodic} if
$M_m(X)=\{\mu\}$. It is easy to see that when $\mu$ is ergodic
every $G$-invariant measurable subset of $X$ has $\mu$ measure $0$
or $1$. Unless we say otherwise we will assume that an $m$-system
is ergodic.


When $\Xcal =(X,\Bcal,G,m,\mu)$ and $\Ycal=(Y,\Acal,G,m,\nu)$
are two $m$-dynamical systems, a measurable map $\pi : X \to Y$
which intertwines the $G$-actions and satisfies $\pi_*(\mu)=\nu$
is called a {\bf homomorphism} of $m$-stationary systems.
We then say that $\Ycal$ is a {\bf factor} of $\Xcal$, or that $\Xcal$
is an {\bf extension} of $\Ycal$.

\br

Let $\Omega=G^\N$ and let $P=m^\N=m\times m \times m \dots$ be the
product measure on $\Omega$, so that $(\Om,P)$ is a probability space.
 We let $\xi_n:\Om\to G$, denote the
projection onto the $n$-th coordinate,\ $\ n=1,2,\dots$. We refer
to the stochastic process $(\Om,P,\{\eta_n\}_{n\in\N})$, where
$\eta_n=\xi_1\xi_2\cdots\xi_n$ as the {\bf $m$-random walk on $G$}.


A real valued function $f(g)$ for which $\int f(gg')\,dm(g')=f(g)$
for every $g \in G$ is called {\bf harmonic}. For a harmonic $f$ we have
\begin{gather*}
E(f(g\xi_1\xi_2\cdots\xi_n\xi_{n+1}|\xi_1\xi_2\cdots\xi_n)\\
=\int f(g\xi_1\xi_2\cdots\xi_n g')\,dm(g') \\ =
f(g\xi_1\xi_2\cdots\xi_n),
\end{gather*}
so that the sequence $f(g\xi_1\xi_2\cdots\xi_n)$ forms a {\bf martingale}.

For $F \in C(X)$ let $f(g)= \int F(gx)\,d\mu(x)$, then the equation 
$m*\mu=\mu$ shows that $f$ is harmonic. 
It is shown
(e.g.) in \cite{F1} 
how these facts combined with the martingale convergence theorem lead to the following:

\begin{thm}
The limits
\begin{equation}\label{1}
\lim_{n\to\infty}\eta_n\mu=\lim_{n\to\infty}\xi_1\xi_2\cdots\xi_n\mu=\mu_\omega,
\end{equation}
exist for $P$ almost all $\om\in\Om$. 
\end{thm}

The measures $\mu_\om$ are the {\bf conditional measures} of the 
$m$-system $\Xcal $.
We let $\Om_0$ denote the subset of $\Om$ where the
limit \eqref{1} exists. The  fact that $\mu$ is $m$-stationary can be
expressed  as:
$$
\int \xi_1(\om)\mu dP(\om)=m*\mu=\mu.
$$
By induction we have
$$
\int \xi_1(\om)\xi_2(\om)\cdots\xi_n(\om)\mu dP(\om)=\mu,
$$
and passing to the limit we also have the {\bf barycenter equation\/}:
\begin{equation}\label{2}
\int\mu_\om dP(\om)=\mu.
\end{equation}

There is a natural ``action" of $G$ on $\Om$ defined as follows. For
$\om =(g_1,g_2,g_3,\dots)\in \Om$ and $g\in G,\ g\om\in\Om$ is
given by $g\om=(g,g_1,g_2,g_3,\dots)$. (This is not an action in the usual
sense; e.g. $g^{-1}(g\om)\ne \om$.)
It is easy to see that for every $g\in G$ and
$\om\in \Om_0$,\ $\mu_{g\om}=g\mu_\om$, so that $\Om_0$ is
$G$-invariant. The map $\zeta:\Om\to M(X)$ given $P$ a.s.  by
$\om\mapsto \mu_\om=\lim_n \xi_1\xi_2\cdots\xi_n\mu$, sends the
measure $P$ onto a probability measure, $\zeta_* P=P^*\in M(M(X))$;\
i.e. $P^*$ is the distribution of the $M(X)$-valued random variable
$\zeta(\om)=\mu_\om$. Clearly for each $k\ge 1$, the
random variable $\zeta_k=\lim_{n\to\infty}
\xi_k\xi_{k+1}\cdots\xi_{k+n}\mu$ has the same distribution $P^*$
as $\zeta(\om)$. We also have $\zeta_k=\xi_k\zeta_{k+1}$. The
functions $\{\zeta_k\}$ therefore satisfy:
\begin{enumerate}
\item[(a)]\
$\zeta_k$ is a function of $\xi_k,\xi_{k+1},\dots$
\item[(b)]\
all the $\zeta_k$ have the same distribution,
\item[(c)]\
$\xi_k$ is independent of $\zeta_{k+1},\zeta_{k+2},\dots$
\item[(d)]\
$\zeta_k=\xi_k\zeta_{k+1}$.
\end{enumerate}
In other words, the $M(X)$-valued stochastic process $\{\zeta_k\}$
is an {\bf $m$-process} in the sense of definition 3.1 of \cite{F1} and it
follows that the measure $P^*$ is $m$-stationary (condition (d))
and that $\Pi(\Xcal )=(M(X),G,m,P^*)$ is an $m$-system 
\footnote {The ``barycenter" equation \eqref{2} is what makes the ``quasifactor"
$\Pi(\Xcal)$ meaningful in the general measure theoretical
setup, where $X$ is just a standard Borel space; see e.g. \cite{G}}.

\br

{\bf Definitions:} We call the $m$-system $\Xcal =(X,G,m,\mu)$,
{\bf $m$-proximal} (or a ``boundary" in the terminology of \cite{F1})
if $P$ a.s. the conditional measures $\mu_\om\in M(X)$ are point
masses. 
Clearly a factor of a proximal system is proximal as well.
Let $\pi: (X,G,m,\mu)\to (Y,G,m,\nu)$ be a homomorphism of
$m$-dynamical systems. We say that $\pi$ is a {\bf measure
preserving homomorphism} (or extension) if for every $g\in G$ we
have $g\mu_y=\mu_{gy}$ for $\nu$  almost all $y$. Here the
probability measures $\mu_y\in M(X)$ are those given by the
disintegration $\mu=\int \mu_y d\nu(y)$.  It is easy to see that
when $\pi$ is a measure preserving extension then also (with
obvious notations),  $P$ a.s. $g(\mu_\om)_y= (\mu_\om)_{gy}$
 for $\nu$  almost all $y$  . Clearly, when $\Ycal $ is the
trivial system, the extension $\pi$ is measure preserving iff the
system $\Xcal $ is measure preserving. We say that $\pi$ is an
{\bf $m$-proximal homomorphism} (or extension) if $P$ a.s. the
extension $\pi:(X,\mu_\om)\to (Y,\nu_\om)$ is a.s. 1-1, where
$\nu_\om$ are the conditional measures for the system $\Ycal $.
Clearly, when $\Ycal $ is the trivial system, the extension $\pi$
is $m$-proximal iff the system $\Xcal $ is $m$-proximal. When
there is no room for confusion we sometimes say proximal rather
than $m$-proximal.

\br

Proposition 3.2 of \cite{F1} can now be formulated as:

\begin{prop}\label{1.1}
For every $m$-dynamical system $\Xcal$ the system
$\Pi(\Xcal )=(M(X)\allowbreak,P^*)$ is
$m$-proximal. It is a trivial, one point system, iff $\Xcal $ is
a measure preserving system.
\end{prop}

\br

Given the group $G$ and the probability measure $m$, there exists a
unique universal $m$-proximal system $(\Pi(G,m),\eta)$ called the
{\bf Poisson boundary} of the pair $(G,m)$. Thus every $m$-proximal
system $(X,\mu)$ is a factor of the system $(\Pi(G,m),\eta)$.

Given an $m$-system $(X,\mu)$ let
$$
h_m(X,\mu)=-\int_G \int_X \log\bigl(\frac{dg\mu}{d\mu}(x)
\bigr ) d\mu(x)dm(g),
$$
or
$$
h_m(X,\mu)=-\sum m(g)\int_X \log\bigl(\frac{dg\mu}{d\mu}(x)
\bigr ) d\mu(x),
$$
when $G$ is discrete. This nonnegative number is the $m$-{\bf
entropy} of the $m$-system $(X,\mu)$.
We have the following theorem (see \cite{F63}, \cite{NZ2}).

\begin{thm}\label{NZ}
\begin{enumerate}
\item
The $m$-system $(X,\mu)$ is measure preserving iff $h_m(X,\mu)=0$.
\item
More generally, an extension of $m$-systems $\pi:(X,\mu)\to
(Y,\nu)$ is a measure preserving extension iff
$h_m(X,\mu)=h_m(X,\nu)$.
\item
An $m$-proximal system $(X,\mu)$ is isomorphic to the Poisson
system $(\Pi(G,m),\allowbreak\eta)$ iff
$$
h_m(X,\mu)=h_m(\Pi(G,m),\eta).
$$
\end{enumerate}
\end{thm}

Typically the conditional measures $\mu_\om$ are singular
to the measure $\mu$. In fact we have the following statement.

\br

\begin{thm}\label{1.5} Let $\Xcal =(X,G,\mu)$ be an
$m$-system with the property that a.s. the conditional measures
$\mu_\om$ are absolutely continuous with respect to $\mu$ \
($\mu_\om<<\mu$).
Then $\mu$ is $G$-invariant;\ i.e. $\Xcal $ is measure preserving.
\end{thm}

\begin{proof}
We consider the usual unitary representation of $G$ on
$H=L_2(X,\mu)$ given by
$$
U_gf(x)=f(g^{-1}x)u(g^{-1},x),\qquad \text{\rm with}\qquad
u(g,x)=\sqrt{\frac{dg^{-1}\mu}{d\mu}}.
$$
For $\om\in\Om_0$ let $f_\om=\frac{d\mu_\om}{d\mu}\in L_1(\mu)$
denote the Radon-Nikodym derivative of $\mu_\om$ w.r.t. $\mu$, and
put $h_\om=\sqrt{f_\om}$.
Then for $\om\in\Om_0,\ g\in G$ and $f\in L_2(X,\mu)$, denoting
$v(g,x)= \frac{dg^{-1}\mu}{d\mu}$, we get
\begin{align*}
\int  f(x)dg\mu_\om(x)&=\int f(gx)f_\om(x)d\mu(x)\\
&= \int f(x)f_\om(g^{-1}x)dg\mu(x)\\
&=\int f(x)f_\om(g^{-1}x)v(g^{-1},x)d\mu(x).
\end{align*}
Hence  $f_{g\om}=(f_\om\circ g^{-1})\cdot v(g^{-1},\cdot)$ and
$$
h_{g\om}=(h_\om\circ g^{-1})\cdot u(g^{-1},\cdot)=U_gh_\om.
$$
It is now easy to see that the map $\mu_\om\mapsto h_\om$
from $\Om_0$ into the unit ball $B$ of $H=L_2(X,\mu)$, is a Borel
isomorphism which intertwines the $G$-action on $\Om_0$
with the unitary action of $G$ on $B$.
If we let $Y$ be the weak closure of the set of functions
$\{h_\om:\om \in \Om_0\}$ in $B$,
we get a compact $G$-space $(Y,G)$ by restricting the unitary
representation $g\mapsto U_g$ to $Y$. Such a $G$-space is WAP
and our theorem follows from theorem \ref{thm5.4} in section \ref{sec-wap} 
below, which asserts that every $m$-stationary measure on $Y$ is
$G$-invariant. 
(For the definition and basic properties of {\bf
weakly almost periodic} (WAP) $G$-systems we refer e.g. to 
\cite[Chapter 1]{G}.)
\end{proof}

\section{Examples}

{\bf 1.}\ Let $G=SL(2,\R)$ and let $m$ be any
absolutely continuous right and left $K$ invariant probability
measure on $G$ such that $\supp(m)$ generates $G$ as a semigroup.
$G$ acts on the compact space $X$ of rays emanating from the origin
in $\R^2$---which is homeomorphic to the unite circle in $\R^2$.
Normalized Lebesgue measure $\mu$ is the unique $m$-stationary
measure on $X$. $G$ acts as well on the space $Y=\mathbb{P}^1$ of lines
in $\R^2$ through the origin (the projective line) and the natural
map $\pi:X\to Y$, that sends a ray in $X$ to the unique line that
contains it in $Y$, is a 2 to 1 homomorphism of $m$-systems, where
we take $\nu=\pi(\mu)$. It is easy to see that $(Y,\nu)$ is
$m$-proximal and that $\pi$ is a measure preserving extension. It
can be shown that $(Y,\nu)$ is the unique $m$-proximal system so
that in particular $(Y,\nu)$ is the Poisson boundary $\Pi(G,m)$.

\br

{\bf 2.}\ (\cite{F})
Let $G$ be a connected semisimple Lie group with finite center and
no compact factors.
Let $G=KNA$ be an Iwasawa decomposition,
$S=AN$ and $P=MAN$, the corresponding minimal parabolic
subgroup. Set $X=G/S$, $Y=G/P$ and let $m$ 
be an admissible probability measure on $G$.
More specifically we assume that $m$ is absolutely continuous
with respect to Haar measure, right and left $K$-invariant,
and $\supp (\mu)$ generates $G$ as a semigroup. Then
\begin{enumerate}
\item
There exists on $Y$ a unique $m$-stationary measure $\nu$ (which is
the unique $K$-invariant probability measure on $Y$) such that the
$m$-system $(Y,\nu)$ is $m$-proximal. In fact $(Y,\nu)$ is the
Poisson boundary $\Pi(G,m)$ and the collection of $m$-proximal
systems coincides with the collection of homogeneous spaces $G/Q$
with $Q$ a parabolic subgroup of $G$.
\item
For any $m$-stationary measure $\mu$ on $X$ the
natural projection $(X,\mu)\overset\pi\to (Y,\nu)$
is a measure preserving extension.
\end{enumerate}

\br

{\bf 3.}\ (\cite{NZ1}) Let $G$ be a connected semisimple Lie group with
finite center, no compact factors, and $\R$-rank $(G)\ge2$. Let $m$
be an admissible probability measure on $G$ and let $(X,G)$ be
a compact metric $G$-space. Let $P$ be
a minimal parabolic subgroup of $G$ and $\la$ a $P$-invariant
probability on $X$. Let $\nu_0$ be the unique $m$-stationary
probability measure on $G/P$. Let $\tilde\nu_0$ be any probability
measure on G which projects onto $\nu_0$ under the natural
projection of $G$ onto $G/P$, and put $\mu=\tilde\nu_0*\la$ (it
follows from \cite{F}, that $(X,\mu)$ is an $m$-system, and moreover
that any $m$-stationary measure on $X$ is of this form). Suppose
further that the measure preserving $P$-action $(X,\la)$ is mixing.
Then there exists a parabolic subgroup $Q\subset G$, a $Q$-space
$Y$, and a $Q$-invariant probability measure $\eta$ on $Y$ such
that the $m$-system $(X,\mu)$ is isomorphic to the ``induced"
$m$-system $Y\underset Q\times G/Q=((Y\times G)/Q,\tilde\eta)$,
where $\tilde\eta$ is an $m$-stationary measure. In particular
$(X,\mu)$ is a measure preserving extension of an $m$-proximal
system $G/Q$, and $\mu$ is $G$-invariant iff $Q=G$.

\br

In the following examples let $G$ be the free group on two
generators, $G=F_2=\langle a,b \rangle$, and $m=\frac
14(\del_a+\del_b+\del_{a^{-1}}+\del_{b^{-1}})$.

\br

{\bf 4.}\ (See \cite{F1})
Let $Z$ be the space of right infinite reduced words on the letters
$\{a,a^{-1},b,b^{-1}\}$. $G$ acts on $Z$ by concatenation on the left
and reduction. Let $\eta$ be the probability measure on $Z$ given by
$$
\eta(C(\ep_1,\dots,\ep_n))=\frac1{4\cdot 3^{n-1}},
$$
where for $\ep_j\in\{a,a^{-1},b,b^{-1}\}$,\
$C(\ep_1,\dots,\ep_n)=\{z\in Z:z_j=\ep_j,\ j=1,\dots,n\}$. The
measure $\eta$ is $m$-stationary and the $m$-system $\Zcal
=(Z,\eta)$ is $m$-proximal. In fact $\Zcal $ is the Poisson
boundary $\Pi(F_2,m)$.

\br

{\bf 5.}\ 
Let $Y=\{0,1\},\ \nu=\frac 12(\del_0+\del_1)$, and the
action be defined by $a\ep=\bar\ep,\ b\ep=\bar\ep$ for $\ep\in
\{0,1\}$, where $\bar 0=1$ and $\bar 1=0$. $\Ycal =(Y,\nu)$ is a
measure preserving system.

\br

{\bf 6.}\
Let $X=Y\times Z$,\ $\mu=\nu\times\eta$,\
where $Y,Z,\nu,\eta$ are as above, and let the action
of $G$ on $X$ be defined as follows:
\begin{align*}
a(\ep,z)&=(\bar\ep,a_\ep z),\qquad a^{-1}(\ep,z)
=(\bar\ep,a^{-1}_{\bar\ep} z),\\
b(\ep,z)&=(\bar\ep,b_\ep z),\qquad b^{-1}(\ep,z)
=(\bar\ep,b^{-1}_{\bar\ep} z),
\end{align*}
where for $g\in G$ we
let $g_0=e$ and $g_1=g$. Finally let $\pi:X\to Y$ be the
projection on the first coordinate. One can check that $m*\mu=\mu$
so that $\Xcal $ is an $m$-system, and that the extension $\pi$
is a relatively proximal extension. We claim that the following
system is a description of $\Pi(\Xcal )=(M,P^*)$. Let
$M=\{\langle (\ep,z),(\bar\ep, z')\rangle: \ep\in \{0,1\},\
z,z'\in Z\}$, here $\langle\cdot,\cdot\rangle$ denotes the {\it
unordered} pair. The measure $P^*$ is given by
$$
P^*\bigl((\{\ep\}\times A)\times (\{\bar\ep\}\times B)\cup
(\{\bar\ep\}\times B)\times (\{\ep\}\times A)\bigr)=\eta(A)\eta(B),
$$
for $A,B\subset Z$ and $\ep\in \{0,1\}$. It is not hard to see
that, although the $m$-system $\Xcal $ is not measure preserving,
it admits no nontrivial $m$-proximal factor.

\br

{\bf 7.}\ 
A small variation on example 6 gives an example of a
similar nature, with the conditional measures $\mu_\om$ being
continuous. Take $Y$ to be the diadic adding machine
$Y=\{0,1\}^\N= \{\ep=(\ep_1,\ep_2,\ep_3,\dots):\ep_i\in\{0,1\}\}$,
let $X=Y\times Z$, and define the action of $F_2$ on $X$ by:
\begin{align*}
 a(\ep,z)&=(\ep+\ch,a_\ep z),\qquad
a^{-1}(\ep,z)=(\ep+\ch,a^{-1}_{\ep} z),\\
b(\ep,z)&=(\ep+\ch,b_\ep z),\qquad
b^{-1}(\ep,z)=(\ep+\ch,b^{-1}_{\ep+\ch} z),
\end{align*}
where $\ch=(1,0,0,\dots)$ and
$a_\ep=e$ when $\ep_1=0$,\  $a_\ep=a$ when $\ep_1=1$,
and $b_\ep$ is defined similarly.

\br

{\bf 8.}\
Let $G$ be the closed subgroup of the Lie group
$GL(4,\R)$ consisting of all $4 \times 4$ matrices of the form
$$
\begin{pmatrix}
A & 0 \\
0 & B
\endpmatrix
\qquad\text{and} \qquad
\pmatrix
0 & A \\
B & 0
\end{pmatrix}
$$
with $A,B\in GL(2,\R)$. We let $G$ act on the subspace $X$ of the
projective space $\mathbb{P}^3$ consisting of the disjoint union of the two
one dimensional projective spaces $\mathbb{P}^1$, which are naturally
embedded in $\mathbb{P}^3$, the quotient space of
$\R^4=\R^2\times \R^2$. Call these two copies $X_1$ and $X_2$ respectively.
There is a natural projection from $(X,G)$
onto the two-point $G$-system $(Y,G)=(\{X_1,X_2\},G)$.
Let $m$ be an admissible probability on $G$ and $\mu$ an
$m$-stationary measure on $X$. Then it is easy to see that
the $m$-system $(X,\mu)$ is an $m$-proximal extension of the
(measure preserving) two-point system $Y$.
Moreover the $m$-system
$(X,\mu)$ has no nontrivial $m$-proximal factor.
If we let  $Z\subset M(X)$ be the collection of measures of the form:
$$
Z=\{\frac 12 (\del_{x_1}+\del_{x_2}): x_i\in X_i,\ i=1,2\},
$$
then one can check that the elements of $Z$ are the conditional measures
$\mu_\om$ of the $m$-system $(X,\mu)$.
It follows that $(M(X),P^*)$ is isomorphic as an $m$-system to the
symmetric product 
$\mathbb{P}^1\times \mathbb{P}^1/ \{\id,\text{flip}\}$.

\br


{\bf 9.}\ (\cite{NZ2})
Let $G=SL(2,\R)$ and fix an admissible $K$-invariant measure $m$ on $G$.
In \cite[Theorem 3.1]{NZ2} Nevo and Zimmer construct a co-compact lattice 
$\Ga< G=SL(2,\R)$, a $\Ga$-space $Z$ and an $m$-stationary measure
$\eta$ on the induced $G$-space $X=G/\Ga\underset {\Ga }{\times} Z$,
with the property that $0 < h_\eta(X) < h_\nu(Y)$, where $Y=\Pi(G,m)$ and
$\nu$ is the unique $m$-stationary probability measure on $Y$
(see example ${\bf{1}}$ above).

\br

{\bf Claim:} The $m$-system $(X,\eta,G)$ admits no nontrivial $m$-proximal factors.

\begin{proof} 
There is a unique $m$-proximal $G$-system, namely the
Poisson boundary $(\Pi(G,\allowbreak m),\nu)$. Since the entropy of the
$m$-system $(G/\Ga\underset {\Ga} {\times} Z,\eta,G)$ is strictly
lower than the entropy of $(\Pi(G,m),\nu)$, the former cannot admit
the latter as a factor.
\end{proof}

\section{Joinings}

{\bf Definitions:} Let $\Xcal $ and $\Ycal $ be two $m$-systems.
We say that a probability measure $\la$ on $X\times Y$ is an {\bf
$m$-joining} of the measures $\mu$ and $\nu$ if it is
$m$-stationary and its marginals are $\mu$ and $\nu$ respectively.
In contrast to the situation in the class of measure preserving
dynamical systems, the product measure $\mu \times \nu$ is usually
not $m$-stationary and therefore not an $m$-joining. On the other
hand we have the following natural construction. We let the
probability measure $\la \in M(X\times Y)$ be defined by
$$
\la=\mu\curlyvee\nu=\int \mu_\om \times \nu_\om dP(\om).
$$
The equation
$$
g\la=\int \mu_{g\om} \times \nu_{g\om} dP(\om),
$$
for each $g\in G$, implies
\begin{align*}
\int g\la dm(g)&=\int\int g\mu_\om \times g\nu_\om dP(\om)
dm(g)\\
&= \int\int \mu_{g\om} \times \nu_{g\om} dP(\om)dm(g)\\
&=\int \mu_\om \times \nu_\om dP(\om)=\la;
\end{align*}
i.e.\ $\la$ is $m$-stationary. We call the $m$-system $\Xcal
\curlyvee \Ycal =(X\times Y,\la)$, the $m$-{\bf join} of the two
$m$-systems $\Xcal $ and $\Ycal $. We use the notation $\Xcal
\vee \Ycal $ to denote any joining of the systems $\Xcal $ and
$\Ycal $; e.g. when they are both factors of a third $m$-system
$\Zcal $ then we usually mean $\Xcal  \vee \Ycal $ to be the
factor of $\Zcal $ defined by the smallest $\sig$-algebra
containing $\Xcal $ and $\Ycal $.

\br

\begin{prop}\label{1.3}
Let $\Xcal $ and $\Ycal $ be two
$m$-systems,
\begin{enumerate}
\item
if $\Xcal $ is measure preserving then $\mu \curlyvee \nu=\mu
\times \nu$;
\item
if $\Xcal $ is $m$-proximal then
$$
\mu \curlyvee\nu=\int \del_{x_\om}\times \nu_\om dP(\om)
$$
is the unique $m$-joining of the two systems.
\end{enumerate}
\end{prop}

\begin{proof}
(1)\  Since the conditional measures for $\Xcal $
satisfy $\mu_\om=\mu$ a.s.,
\begin{align*}
\mu\curlyvee\nu&=\int \mu_\om \times \nu_\om dP(\om)\\
&\int \mu \times \nu_\om dP(\om)\\
&=\mu\times \int \nu_\om dP(\om)\\
&=\mu\times\nu.
\end{align*}

(2)  \ Let $\la$ be any $m$-joining of $\mu$ and $\nu$. Our
assumption now is that the conditional measures of $\Xcal $ are
a.s. point masses $\del_{x_\om}$, whence the conditional measures
$\la_\om=\lim \xi_1\xi_2\cdots\xi_n\la$ have marginals
$\del_{x_\om}$ and $\nu_\om$ on $X$ and $Y$ respectively. This
means $\la_\om= \del_{x_\om}\times\nu_\om$ and therefore
$$
\la=\int \la_\om dP(\om)=\int \del_{x_\om}\times\nu_\om dP(\om)=
\mu \curlyvee\nu.
$$
\end{proof}

\br

\begin{prop}\label{prop2.2}
\begin{enumerate}
\item
The only endomorphism of a proximal system is the identity
automorphism.
\item
For every $m$-system $(X,\mu)$ there is a unique maximal proximal
factor.
\end{enumerate}
\end{prop}

\begin{proof}
(1)\ Let $\al:X\to X$ be an endomorphism of the proximal system
$(X,\mu)$. Consider the map $\phi:x\mapsto
\tet_x=\frac 12 (\del_x+\del_{\al(x)})$ of $X$ into $M(X)$.
This induces a quasifactor $(M(X),\la)$ where $\la=\phi_*(\mu)$.
Now the conditional measures of the proximal system $(M(X),\la)$
are point masses of the form $\del_{\tet_x}$. On the other hand
applying the barycenter map $b$ to the limits:
$$
\xi_1(\om)\cdots\xi_n(\om)\la\to \del_{\tet_{x(\om)}},
$$
we get
$$
\xi_1(\om)\cdots\xi_n(\om)\mu\to \del_{x(\om)}.
$$
Thus $b(\del_{\tet_{x(\om)}})=\tet_{x(\om)}=\del_{x(\om)}$ a.e.;
i.e. $\al(x)=x$ a.e.

(2)\ It is easy to check that the join of all proximal factors of
an $m$-system $(X,\mu)$ is a maximal proximal factor of $(X,\mu)$.
\end{proof}

\section{A structure theorem for stationary systems}

\begin{prop}\label{2.1}
Let
\begin{equation*}
\xymatrix
{
( X,\mu) \ar[dd]_{\pi} \ar[dr]^{\sig}  &  \\
& (Z,\eta) \ar[dl]^{\rho}\\
(Y,\nu) &
}
\end{equation*}
be a commutative diagram of $m$-systems
\begin{enumerate}
\item
if $\pi$ is a measure preserving extension then so are $\rho$ and $\sig$.
\item
if $\pi$ is a proximal extension then so are $\rho$ and $\sig$.
\end{enumerate}
\end{prop}
\begin{proof}
(1)\ Let
$$
\mu=\int \mu_y d\nu(y)=\int \mu_z\ d\eta(z),\qquad \eta
=\int\eta_y d\nu(y),
$$
be the disintegrations of $\mu$ over $Y$ and $Z$ and of
$\eta$ over $Y$ respectively. We assume that for all $g$ and $\nu$
almost  every $y$,\ $g\mu_y=\mu_{gy}$, hence
$$
g\eta_y=g\sig\mu_y=\sig g\mu_y=\sig\mu_{gy}=\eta_{gy},
$$
so that $\rho$ is a measure preserving extension. Now, since
$$
\mu=\int \mu_y d\nu(y)=\int \mu_z\ d\eta(z)
=\int\bigl(\int \mu_z\ d\eta_y(z)\bigr)d\nu(y),
$$
the uniqueness of disintegration shows that
$$
\mu_y=\int \mu_z\ d\eta_y(z).
$$
Thus for $g\in G$ we have:
$$
g\mu_y=g\bigl(\int \mu_z\ d\eta_y(z)\bigr)
=\int g\mu_z d\eta_y(z),
$$
and also
$$
g\mu_y=\mu_{gy}=\int \mu_z\ d\eta_{gy}(z)
=\int \mu_z\ dg\eta_y(z)= \int \mu_{gz}\ d\eta_y(z).
$$
Again the uniqueness of disintegration yields $g\mu_z=\mu_{gz}$,
so that also $\sig$ is a measure preserving extension.

(2)\ This is a straightforward consequence of the definition of
$m$-proximal extension.
\end{proof}

\br

Let us call an $m$-system $(X,\mu)$ {\bf standard} if there exists
a homomorphism $\pi:(X,\mu)\to(Y,\nu)$ with $(Y,\nu)$ proximal and
the homomorphism $\pi$ a measure preserving extension.
Note that with this terminology the results described in the examples
{\bf{2}} and {\bf{3}} above can be stated as saying that the stationary systems
$(X,\mu)$ described there are standard (of a very particular kind,
namely measure preserving extensions of boundaries of the form
$G/Q$ with $Q \subset G$ a parabolic subgroup). 

\begin{prop}\label{ 2.2}
\begin{enumerate}
\item
The structure of a standard system as a measure preserving
extension of a proximal system is unique.
\item
Let $(X,\mu)$ be a standard $m$-system: $\pi:(X,\mu)\to(Y,\nu)$
with $(Y,\nu)$ proximal and the homomorphism $\pi$ a measure
preserving extension. If $\al:(X,\mu)\to (Z,\eta)$ is a measure
preserving homomorphism then there is a commutative diagram:
\begin{equation*}
\xymatrix
{
X \ar[dd]_{\pi} \ar[dr]^{\al}  &  \\
& Z \ar[dl]^{\beta}\\
Y &
}
\end{equation*}
\end{enumerate}
\end{prop}

\begin{proof} (1)\ Let $(X,\mu)$ be a standard $m$-system:
$\pi:(X,\mu)\to(Y,\nu)$ with $(Y,\nu)$ proximal and the
homomorphism $\pi$ a measure preserving extension. If
$\pi':(X,\mu)\to(Y',\nu')$ is another factor with $(Y',\nu')$
proximal, then the system $Y\vee Y'$ is also $m$-proximal and we
have the diagram:
\begin{equation*}
\xymatrix
{
X \ar[dd]_{\pi} \ar[dr]^{\sig}  &  \\
& Y \vee Y'\ar[dl]^{\rho}\\
Y &
}
\end{equation*}
Now $\rho$ is clearly a proximal extension and by proposition \ref{2.1}
it is also a measure preserving extension. Thus $\rho$ is an
isomorphism, so that $Y'$ is a factor of $Y$. We now have the
diagram:
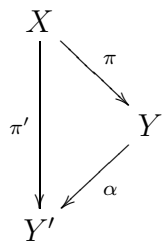
\begin{equation*}
\xymatrix
{
X \ar[dd]_{\pi'} \ar[dr]^{\pi}  &  \\
& Y \ar[dl]^{\al}\\
Y'&
}
\end{equation*}
If $\pi'$ is a measure preserving homomorphism then by proposition
\ref{2.1} so is $\al$ and being also a proximal homomorphism it is
necessarily an isomorphism.

(2)\ Consider the diagram:
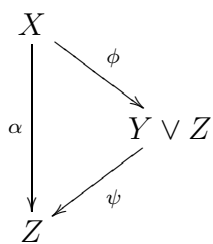
\begin{equation*}
\xymatrix
{
X \ar[dd]_{\al} \ar[dr]^{\phi}  &  \\
& Y\vee Z \ar[dl]^{\psi}\\
Z &
}
\end{equation*}
Since $\al$ is a measure preserving homomorphism so is $\psi$
(proposition \ref{2.1}). On the other hand, since $Y$ is proximal it
follows that $\psi$ is a proximal extension. Thus $\psi$ is an
isomorphism and we deduce that $Y$ is a factor of $Z$ as required.
\end{proof}

\br

\begin{thm}[A structure theorem for stationary systems]\label{2.3}
Let
$\Xcal=(X,\mu)$ be an $m$-system, then there exist canonically
defined $m$-systems $\Xcal^*=(X^*,\mu^*)$, and $\Pi(\Xcal)=(M,P^*)$,
with $\Xcal^*$ standard and $\Pi(\Xcal)$ \
$m$-proximal, and a diagram
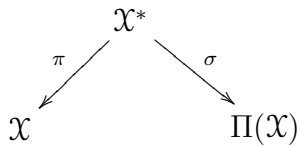
\begin{equation*}
\xymatrix
{
& \Xcal^* \ar[dl]_{\pi} \ar[dr]^{\sig}  &  \\
\Xcal &  & \Pi(\Xcal)
}
\end{equation*}
where $\pi$ is an $m$-proximal extension, and $\sig$ is a measure
preserving extension. Thus every $m$-system admits an $m$-proximal
extension which is standard. The $m$-system $\Xcal$ is measure
preserving iff $\Pi(\Xcal)$ is trivial. The $m$-system $\Xcal$
is $m$-proximal iff both $\pi$ and $\sig$ are isomorphisms. We
call $\Xcal^*$ the {\bf standard cover} of $\Xcal$.
\end{thm}

\begin{proof}
We let $\Xcal^*=\Xcal \curlyvee \Pi(\Xcal)$.
Thus $X^*=X\times M(X)$ and the measure $\mu^*=\mu \curlyvee P^*$
is defined by the integral
\begin{equation}\label{3}
\mu^*=\int \mu_\om\times\del_{\mu_\om} dP(\om).
\end{equation}
The assertions of the theorem now follow from propositions
\ref{1.1} and
\ref{1.3}, \ 
however for clarity and completeness we give below a more
detailed proof. Denote by $\pi$ and $\sig$ the projections on
the first and second coordinates respectively.
Clearly $\Xcal^*=(X^*,\mu^*)$ is a joining of the systems $(X,\mu)$ and
$(M(X),P^*)$ in the sense that $\pi(\mu^*)=\mu$, and
$\sig(\mu^*)=P^*$. 
We show next that $\mu^*$ is $m$-stationary.
For $g\in G$ we have a.s.
\begin{equation}\label{4}
g\mu_{\om}=\mu_{g\om},
\end{equation}
hence
$$
g\mu^*=\int \mu_{g\om}\times\del_{\mu_{g\om}} dP(\om),
$$
hence
\begin{align*}
\int_G g\mu^* dm(g)&=\int_G\int_\Om \mu_{g\om}
\times\del_{\mu_{g\om}} dP(\om) dm(g)\\
&=\int \mu_{\xi_1 \om}\times\del_{\mu_{\xi_1 \om}} dP(\om)\\
&=\int \mu_\om\times\del_{\mu_\om} dP(\om)=\mu^*.
\end{align*}
Now \eqref{3} gives the disintegration of $\mu^*$ with respect to $P^*$,
i.e. w.r.t. $\sig$, and \eqref{4} shows that $\sig$ is a measure preserving
extension.

Next we mimic the proof of proposition 3 in \cite{F1} in order to show that the
measures $\tet_\om=\mu_\om\times\del_{\mu_\om}$ are the conditional
measures of the $m$-system $( X^*,\mu^*)$, \ i.e. we will show that
a.s.
\begin{equation}\label{5}
\lim \xi_1\xi_2\cdots\xi_n \mu^*=\tet_\om.
\end{equation}
First observe that
$$
\tet_\om=\lim_{n\to\infty}
\xi_1\xi_2\cdots\xi_n (\mu\times\del_\mu).
$$
Write $\tet_1(\om):=\tet_\om$ and let
$$
\tet_k=\lim_{l\to\infty}\xi_k\xi_{k+1}\cdots \xi_{k+l}
(\mu\times\del_\mu),
$$
so that $\xi_1\xi_2\cdots\xi_n\tet_{n+1}=\tet_1$.
For a bounded continuous function $f$ on $X^*$ and a measure
$\iota\in M(X^*)$ we write $f(\iota)=\int_{X^*}f(x^*)d\iota(x^*)$.
Now for any such $f$ we have
\begin{align*}
&\int_{X^*} f(\xi_1\xi_2\cdots\xi_nx^*) d\mu^*(x^*)\\
=&\int_\Om \int_X f(\xi_1\xi_2\cdots\xi_n(x,\del_{\om'})
d\mu_{\om'}(x)dP(\om')\\
=&\int_\Om f(\xi_1\xi_2\cdots\xi_n(\mu_{\om'}\times\del_{\del_{\om'}})
dP(\om')\\
=&E\bigl( f(\xi_1\xi_2\cdots\xi_n(\tet_{n+1})|\xi_1\xi_2\cdots\xi_n\bigr)\\
=& E\bigl( f(\tet_1)|\xi_1\xi_2\cdots\xi_n\bigr) \to f(\tet_1)=
f(\mu_\om\times\del_{\mu_\om}),
\end{align*}
where the convergence in the last line follows from the martingale
convergence theorem. Since clearly a.s.
$\pi:(X\times M,\mu_\om\times\del_{\mu_\om}) \to (X,\mu_\om)$
is 1-1, we see that $\pi$ is an $m$-proximal extension.
This completes the proof of the theorem.
\end{proof}

\br

\begin{thm}\label{2.4}
If $(X,\mu)$ is an $m$-system with
maximal entropy (i.e. $h_m(X,\mu)=h_m(\Pi(G,m))$); then $(X,\mu)$
is standard and it admits the Poisson boundary $\Pi(G,m)$ as its
maximal proximal factor.
\end{thm}
\begin{proof}
Let $\pi:X^*\to X$ be the standard cover of $(X,\mu)$,
so that in particular $\pi$ is a proximal extension. Since $\pi$
does not raise entropy it is a measure preserving extension
(theorem \ref{NZ}). Thus $\pi$ is an isomorphism and $X$ is a
standard system whose maximal proximal factor has maximal entropy.
Again theorem \ref{NZ} implies that this factor is isomorphic to
$\Pi(G,m)$.
\end{proof}

\br

\begin{thm}\label{stand}
Let $\Xcal=(X,\mu)$ be an $m$-system which admits a strict tower
of proximal and measure preserving extensions
$$
X \cdots \to X_{n+1} \to X_n \to \cdots \to X_2 \to X_1 \to X_0.
$$
Assume that $X_0$ is the maximal proximal factor of $\Xcal$,
that the proximal and measure preserving maps
alternate, and that each such map is maximal.
Thus $X_{2n+1}\to X_{2n}$ is measure preserving and
if $X \to Y \to X_{2n+1}\to X_{2n}$ is such that
$Y \to X_{2n}$ is also measure preserving then $Y =X_{2n+1}$.
Likewise
$X_{2n+2}\to X_{2n+1}$ is proximal and
if $X \to Y \to X_{2n+2}\to X_{2n+1}$ is such that
$Y \to X_{2n+1}$ is also proximal then $Y =X_{2n+2}$.
Let
\begin{equation}\label{pi}
\Pi(X) \cdots \to \Pi(X_{n+1}) \to
\Pi(X_n) \cdots \to \Pi(X_2) \to \Pi(X_1) \to \Pi(X_0)=X_0
\end{equation}
be the corresponding sequence of homomorphisms. (Since
for every $n$ the map $\phi: X_{2n+1}\to X_{2n}$ is measure preserving,
$\Pi(X_{2n+1})\to \Pi(X_{2n})$ is an isomorphism.)
If at any stage in this sequence we have that
$\Pi(X_{2n+2})\to \Pi(X_{2n+1})$ is an isomorphism,
then $X = X_{2n+1}$.
\end{thm}

\begin{proof}
For convenience we write $m=2n+1$. In the diagram
\begin{equation*}
\xymatrix
{
& X_{m+1}\vee\Pi(X_{m+1})\ar[dd] \ar[dl]_{\prox} \ar[dr]^{\mpr}  &  \\
X_{m+1}\ar[dd]_{\prox} &  & \Pi(X_{m+1})\ar[dd]^{\phi}\\
&   X_m\vee \Pi(X_m)\ar[dl]_{\prox}\ar[dr]^{\mpr} & \\
X_m  &                             & \Pi(X_m)
}
\end{equation*}
by assumption, $\phi$ is an isomorphism and therefore all
the maps on the right of the central vertical arrow
$X_{m+1}\vee\Pi(X_{m+1}) \to X_m\vee \Pi(X_m)$
are measure preserving maps. On the other hand all the arrows
on the left of this arrow are proximal maps. We conclude that
$X_{m+1}\vee\Pi(X_{m+1}) \to X_m\vee \Pi(X_m)$ is both
measure preserving and proximal, hence an isomorphism.
However this implies that also $X_{m+1} \to X_m$ is
an isomorphism. Since we assumed that at each stage the
extension is maximal we now realize that the whole tower
above $X_{m}$ collapses, i.e. $X=X_m=X_{2n+1}$.
\end{proof}

\br

\begin{cor}
For $G=SL(n,\R)$ and $K$-invariant admissible $m$,
every strict maximal tower is of height $\le n$.
\end{cor}

\begin{proof}
As was shown in \cite{F} the Poisson $(G,m)$-space
$\Pi(G,m)$ is the flag manifold on $\R^n$.
Since every proximal $G$-system is a factor of $\Pi(G,m)$,
every sequence of the form \eqref{pi} is defined by a nested
sequence of parabolic subgroups, whence of length at most $n$.
\end{proof}

\br

{\bf Examples:} {\bf 10.}\ Applying the construction of the
structure theorem to example {\bf 3.} in section 1, we obtain the
following description for the $m$-system $(X^*,\mu^*)=(X\times
M,\mu \curlyvee \nu)$. $X^*$ can be taken as the subset of $X\times
X$ consisting of all ( {\it ordered}\ ) pairs $((\ep,z),(\bar\ep,
z')),\ \ep\in\{0,1\},\ z,z'\in Z$, with the diagonal action
$g((\ep,z),(\bar\ep, z'))= (g(\ep,z),g(\bar\ep, z'))$. The measure
$\mu^*$ is then given by
$$
\frac 12 (\del_0 +\del_1)\times\eta\times\eta.
$$

\br

{\bf 11.}\ As we have seen (proposition \ref{prop2.2}), for every $m$-system
$(X,\mu)$ there is a uniquely defined maximal proximal factor. This
is not always the case with respect to measure preserving factors.
We produce next an example of a product system $(X,\mu)=(Z\times
Y,\eta\times \nu)$ where $(Z,\eta)$ is $m$-proximal and $(Y,\nu)$
is measure preserving---so that $(X,\mu)$ is standard---with a
factor $(Y',\nu')$ which is also measure preserving but such that
the factor $Y\vee Y'$ of $X$ is not measure preserving.

We let $G=F_2$, the free group on two generators $a$ and $b$,
$m=\frac 14(\del_a+\del_b+\del_{a^{-1}}+\del_{b^{-1}})$. Let $Z$ be
the Poisson boundary $\Pi(F_2,m)$ which we can take as the space of
right infinite reduced words on the letters $\{a,a^{-1},b,b^{-1}\}$
with the natural Markov measure $\eta$ as in example (4) above. The
system $(Y,\nu)$ will be the Bernoulli system $Y=\{0,1\}^{F_2}$
with product measure $\nu=\{\frac 12,\frac 12\}^{F_2}$. Thus $\nu$
is an invariant measure under the natural action of $F_2$ on $Y$ by
translations. Clearly $\mu=\eta\times\nu=\eta\curlyvee\nu$ is
$m$-stationary, so that $(X,\mu)$ is an $m$-system.

Next let $A$ be the subset $\{z\in Z: a \ {\text {\rm is\ the\
first\ letter\ of\ }} z\}$, and let $\phi:Z\to Y$ be the continuous
function defined by $(\phi(z))_g=1_A(gz)$. We observe that the map
$\Phi:X\to Y$ defined by $\Phi(z,y)=z+y  \pmod  1 $ is an equivariant
continuous map. Let $Y'=\Phi(X)$ and $\nu'=\Phi_*(\mu)$. It is now
easy to check that $(Y',\nu')$ is a measure preserving factor of
the $M$-system $(X,\mu)$, which is isomorphic to the Bernoulli
system $(Y,\nu)$. However it is also clear that the factor $Y\vee
Y'$ of $(X,\mu)$ is a non-measure preserving $m$-system. In fact
$Y\vee Y'$ admits the non-trivial proximal factor $Z'=\phi(Z)$.

\br

\begin{rmk}
For ergodic probability measure preserving transformations
there is a more satisfactory structure theorem (due to
Furstenberg \cite{Fur4}, \cite{Fur5} and independently to Zimmer \cite{Z1},
\cite{Z2}) according to which every such system is canonically
presented as a weakly mixing extension of a measure-distal
system (the latter is defined as a tower, possibly of infinite height, 
of compact extensions).
In topological dynamics there is an analogous theorem for
a minimal dynamical system $(X,G)$ (see \cite{EGS}, \cite{V}).
However, as in our theorem \ref{2.3},
one is forced in this setting to first associate with $X$ a proximal
extension $X^* \to X$ so that only $X^*$ has the required structure of 
a weakly mixing extension of a PI-system (where the latter
is a tower of alternating proximal and isometric extensions).
In \cite{Gl4} there is an example of a minimal dynamical system
$(X,T)$ which does not admit nontrivial factors that are either proximal
or incontractible (this is the analogue of a measure preserving system
in topological dynamics). We do not know how to construct a similar
example for stationary systems. Such an example will show that
in some sense one can not do better than what one gets in theorem \ref{2.3}.
\end{rmk}

\begin{prob}
Are there a group $G$, a probability measure $m$ on $G$, and
an ergodic $m$-stationary system $\mathcal{X}=(X,\mu,G)$ such that
$\mathcal{X}$ does not admit nontrivial factors that are either proximal
or measure preserving?
\end{prob}

\section{Nevo-Zimmer theorem in an abstract setup}

{\bf Definitions:}
\begin{enumerate}
\item
Notations as in theorem \ref{2.3}, we say that the quasifactor $\Pi
(\Xcal)=(M,P^*)$ is {\bf mixingly embedded} in the $m$-system
$\Xcal=(X,\mu)$, if the measure preserving extension
$\sig:\Xcal^*\to \Pi(\Xcal)$ is a mixing extension;\ i.e. if for every
$f\in L_\infty(\mu^*)$ and every sequence $g_n\to\infty$ in $G$,
$$
w^*{\text -}\lim g_n (f- E^{\Pi(\Xcal)}f)=0,
$$
where $ E^{\Pi(\Xcal)}f$ is the conditional expectation of $f$
with respect to the factor $\Pi(\Xcal)$.
\item
For an  $m$-system $(X,\mu)$ and a subset $V$ of $L_\infty (\mu)$,
let
$$
\Fcal(V)=\{w^*{\text -}\lim g_n f: f\in V,
\ g_n\to\infty\},
$$
the set of all weak$^*$ limit points of sequences $g_nf$ where
$f\in V$ and $g_n\to \infty$.
Let $\Fcal (V)$ be the smallest
$\sig$-algebra with respect to which all members of $\Fcal(V)$ are
measurable. Call the $m$-system $(X,\mu)$\ {\bf reconstructive}
with respect to $V$ if $\Fcal (V)$ is the full $\sig$-algebra of
measurable sets on $X$.
\end{enumerate}

\br

\begin{thm}\label{3.1}
Let $\Xcal=(X,\mu)$ be an $m$-system such that
\begin{enumerate}
\item
The canonical $m$-proximal quasifactor $\Pi(\Xcal)=(M,P^*)$ is
mixingly embedded in $\Xcal$.
\item
$\Pi(\Xcal)$  is a reconstructive $m$-system with respect to the
subspace
$$
V=E^{\Pi(\Xcal)}(C(X)).
$$
\end{enumerate}
Then the $m$-proximal quasifactor $\Pi(\Xcal)$ is actually a
factor of $\Xcal$.
\end{thm}

\begin{proof}
Consider an arbitrary continuous function $f$ on $X$,
$f\in C(X)\subset L_\infty (\mu)\subset
L_\infty(X\times M(X),\mu^*)$, and
the corresponding function
$\tilde f\in L_\infty (P^*)$ on $M(X)$
defined by:
$$
\tilde f(\mu_\om)=\int_X f(x) d\mu_\om= E^{\Pi(\Xcal)}f.
$$
By assumption (1), for every sequence
$g_n\to\infty$ in $G$ for which $w^*{\text -}\lim g_n  \tilde f$ exists,
we have
\begin{equation}\label{equality}
\hat f=w^*{\text -}\lim g_n
f= w^*{\text -}\lim g_n  \tilde f,
\end{equation}
hence $\hat f$ is in the $w^*$-closed subspace $L_\infty(\mu) \cap
L_\infty (P^*)$. 

On the other hand, by assumption (2), with the
subspace $V=\{\tilde f:f\in C(X)\}= E^{\Pi(\Xcal)}(C(X))$, the
smallest $\sig$-algebra with respect to which all the functions:
$$
\{\hat f= w^*{\text -}\lim g_n
\tilde f: f\in C(X),\ g_n\to\infty\}
$$
are measurable is the full $\sig$-algebra of measurable sets on
$\Pi(\Xcal)$. It thus follows that with respect to $\mu^*$,
$L_\infty(P^*)\subset L_\infty(\mu)$ and the proof is complete.
\end{proof}

\br

\begin{cor}
Let $\Xcal=(X,\mu)$ be an $m$-system such that
the canonical $m$-proximal quasifactor $\Pi(\Xcal)=(M,P^*)$ is
mixingly embedded in $\Xcal$. Then for every $f \in
L_\infty(X,\mu)$, for a.e. $\om$
\begin{equation*}
w^*{\text -}\lim \xi_1(\om)\xi_2(\om)\cdots\xi_n(\om)  
 f \equiv  \tilde f (\mu_\om). 
\end{equation*}
\end{cor}
\begin{proof}
In the proof of theorem \ref{3.1}
taking $g_n =\eta_n(\om)=\xi_1(\om)\xi_2(\om)\cdots\xi_n(\om)$
we have, for every $h \in L_1(\mu^*)$ by Lebesgue's dominated
convergence theorem,
\begin{gather*}
\lim_{n\to \infty}\int_{X^*}
\tilde f(\xi_1(\om)\xi_2(\om)\cdots\xi_n(\om)\mu_{\om'})h(x,\mu_{\om'})
\, d\mu^*(x,\mu_{\om'})\\
=\tilde f (\mu_\om)\int_{X^*}h(x,\mu_{\om'})\, d\mu^*(x,\mu_{\om'});
\end{gather*}
i.e. $w^*{\text -}\lim \xi_1(\om)\xi_2(\om)\cdots\xi_n(\om)  
\tilde f \equiv  \tilde f (\mu_\om)$, $\om$-a.s. 
In view of (\ref{equality}) we deduce that $\om$-a.s. 
\begin{equation*}
w^*{\text -}\lim \xi_1(\om)\xi_2(\om)\cdots\xi_n(\om)  
 f \equiv  \tilde f (\mu_\om), 
\end{equation*}

\end{proof}

{\bf Example:} Let $T$ and $S$ be two discrete countable groups,
$m_S$ and $m_T$ probability measures on $S$ and $T$ respectively
such that the corresponding Poisson spaces $\Pi(S,m_S)$ and
$\Pi(T,m_T)$ are nontrivial. We form the product group $G=T\times
S$ and the product measure $m=m_T\times m_S$.

\begin{thm}\label{3.2}
For $G = T \times S$ as above the Poisson spaces for the couples
$(G,m), (T,m_T)$ and $(S,m_S)$ satisfy:
$$
\Pi(G,m)=\Pi(T,m_T)\times \Pi(S,m_S).
$$
\end{thm}

\begin{proof}
Clearly the systems $\Pi(T,m_T)$ and $\Pi(S,m_S)$ can be viewed as
$G$ $m$-systems and as such they are proximal. Thus these systems
are factors of the $m$-system $\Pi(G,m)$. It is now easy to check
that if $\eta_T$ and $\eta_S$ are the $m$-stationary measures on
$\Pi(T,m_T)$ and $\Pi(S,m_S)$ respectively then the measure
$\eta_T\curlyvee\eta_S=\eta_T\times\eta_S$. Whence
$\Pi(T,m_T)\times \Pi(S,m_S)$ is a factor of the system
$\Pi(G,m)$. Since the entropy of both systems is
$h_m(\Pi(T,m_T),\eta_T) + h_m(\Pi(S,m_S),\eta_S)$ we 
can now apply
theorem \ref{NZ} to conclude that $\Pi(G,m)=\Pi(T,m_T)\times
\Pi(S,m_S)$.
\end{proof}

\begin{rmk}
Another proof of this fact follows directly from the characterization
of the Poissson boundary of $(G,m)$ as the space of ergodic
components of the time shift in the path space of the random walk
due to Kaimanovich and Vershik, \cite{KV}.
\end{rmk}

\br

\begin{lem}\label{3.3}
Let $(X,\Bcal,\mu), (Y,\Fcal ,\nu)$ be
two probability spaces, $\Acal $ a sub-$\sig$-algebra of $\Fcal$
and $f\in L_\infty(X\times Y,\mu\times \nu)$. If for $\mu$ a.e.
$x\in X$ the function $f_x(y)=f(x,y)$ is $\Acal $ measurable,
then $f$ is $\Bcal\times \Acal $ measurable.
\end{lem}

\br

\begin{thm}\label{ 3.4}
Let $(X,\mu,G)$ be an $m$-system. If
the canonical $m$-proximal quasifactor $\Pi(\Xcal)=(M(X),P^*)$
is mixingly embedded in $\Xcal$. Then $\Pi(\Xcal)$ is a factor
of $(X,\mu)$.
\end{thm}

\begin{proof}
In view of theorem \ref{3.1}, all we have to show is that
$\Pi(\Xcal)$ is a reconstructive $m$-system with respect to the
subspace $V=E^{\Pi(\Xcal)}(C(X))$. For $f\in C(X)$ the function
$\tilde f$ is $\Pi(\Xcal)$ measurable.  Since $\Pi(\Xcal)$ is
a factor of $\Pi(G)$, by theorem \ref{3.2}, lifting $\tilde f$ to $\Pi(G)$
we can write $\tilde f$ as a function of two variables $\tilde
f(u,v)$, with $u\in \Pi(S)$ and $v\in \Pi(T)$. Now for almost every
$v_0\in \Pi(T)$ there exists a sequence $t^{v_0}_n\in T$ with
$\lim t^{v_0}_n v=v_0$ for $\mu_T$ almost every $v\in \Pi(T)$.
Thus, by (\ref{equality}), we see that
$$
\tilde f^{v_0}(u)=w^*{\text -}\lim t^{v_0}_n  f
    =w^*{\text -}\lim t^{v_0}_n  \tilde f(u,v)
    =\tilde f(u,v_0)
$$
is $\Xcal$ measurable. Similarly for almost every
$u_0\in \Pi(S)$ the function $\tilde f_{u_0}(v)=f(u_0,v)$ is
$\Xcal$ measurable. By lemma \ref{3.3}, $\tilde f$ is $\Xcal$
measurable and since the subspace $V=\{\tilde f:f \in C(X)\}$
generates $C(\Pi(X))$ as an algebra, we conclude that
$\Pi(X)$ is a reconstructive $m$-system with respect to $V$.
\end{proof}

\section{A Szemer\'edi type theorem for $SL(2,\R)$.}
In this section $G$ will denote the Lie group $SL(2,\R)$ and we
write $G=KAN$ for the standard Iwasawa decomposition of $G$; 
in particular $K$ is the subgroup of $2$ by $2$ orthogonal matrices.

Recall that a {\bf mean} on a topological group $G$
is a positive linear functional $\rho$ on $LUC(G)$ with $\rho(\ch)=1$.
Here $LUC(G)$ denotes the commutative $C^*$-algebra of 
bounded, complex valued,
left uniformly continuous functions on $G$. ($f: G \to \C$ is
{\bf left uniformly continuous} if for every $\ep>0$ there exists a neighborhood 
$V$ of the identity element $e\in G$ such that 
$\sup_{g\in G} |f(vg)-f(g)| < \ep$ for every $v \in V$.)
The set of means on $G$ forms a $w^*$-closed convex subset of
$LUC(G)^*$ and we say that an element of this set is {\bf $m$-stationary}
if $m*\rho= \rho$. By the Markov-Kakutani fixed point theorem
the set of $m$-stationary means is nonempty.

Let $Z$ be the (compact Hausdorff) Gelfand 
space corresponding to the  $C^*$-algebra $\mathcal{L}=LUC(G)$. 
Recall that $Z$ can be viewed as the space of non-zero
continuous $C^*$-homomorphisms of the $C^*$-algebra $\mathcal{L}$
into $\C$. In particular, for each $g\in G$ the evaluation
map $z_g: F \mapsto F(g)$ is an element of $Z$.
The fact that $\mathcal{L}$ is $G$-invariant 
(i.e. for $f \in \mathcal{L}$ and $g \in G$ also
$f_g\in \mathcal{L}$, where $f_g(h)=f(gh)$) 
implies that there is a naturally defined $G$-action on $Z$.
We have $g z_e = z_g$ for every $g \in G$, and it follows directly
that the $G$-orbit of the point $z_e$ is dense in $Z$.
Also by the construction of the Gelfand space we obtain a natural isomorphism 
of the commutative $C^*$-algebras $\mathcal{L}$ and $C(Z)$. 
Let $\tilde{f}$ denote the element of $C(Z)$
which corresponds to $f$ under this isomorphism.
Now according to Riesz' representation theorem we identify 
$LUC(G)^*$ with the Banach space of complex regular Borel measures
on $Z$. In this setting a mean on $G$ is identified with a probability 
measure on $Z$. If $L$ is a subset of $G$ and $\rho$ is a mean on
$G$, we say that $L$ is {\bf charged by $\rho$} and write $\rho(L)>0$ if 
$$
\mu_\rho(\cls \{z_g: g \in L\}) > 0.
$$
Here $\mu_\rho$ is the probability measure on $Z$ which corresponds
to the mean $\rho$. 
It is easy to check that with respect to the natural $G$-action on $Z$
the mean $m*\rho$ 
(defined by $m*\rho(f) = \int_G \rho(f_g)\,dm(g)$)
corresponds to the measure $m*\mu_\rho$, so that 
$\rho$ is $m$-stationary if and only if the measure $\mu_\rho$ is $m$-stationary.

\begin{thm}\label{sl}
Let $m$ be an admissible probability measure on $G$
and let $\rho$ a $K$-invariant $m$-stationary mean on
$G=SL(2,\R)$. If $\rho(L) >0$
for a subset $L\subset G$ then for every $\ep>0$, $k\ge 1$ 
and a compact set $Q \subset G$,
there exist $a_0$ and $h$ in $G \setminus Q$ such that
$$
a_0, h a_0,\dots,h^k a_0\in L_\ep,
$$
where $L_\ep=\{g \in G: d(g,L)<\ep\}$.
\end{thm}

\begin{lem}\label{corres}(A correspondence principle)
Let $G$ be a locally compact group.
Given a nonempty subset $L \subset G$,
there exists a compact metric $G$-space $X$ 
and an open subset $A \subset X$ such that 
$$
g^{-1}_1 A\cap g^{-1}_2 A \cap\cdots \cap g^{-1}_k A \ne \emptyset
\Longrightarrow g_1 a, \dots,g_k a \in L_\ep,
$$
for some $a \in G$.
If, moreover, $m$ is a probability measure on $G$ and 
$\rho$ an $m$-stationary mean on $G$ with $\rho(L) >0$,
then there exists an $m$-stationary probability measure
$\mu$ on $X$ with $\mu(A) \ge \rho(L)$.
\end{lem}

\begin{proof}
Let $f: G \to [0,1]$ be a left uniformly continuous function
such that $f(g) = 1$ for every $g \in L_{\ep/2}$ and
$f(g) =0$ for every $g \not\in L_\ep$.
Let $\mathcal{A}$ be the uniformly closed subalgebra of
the algebra $LUC(G)$ of complex valued bounded left uniformly continuous
functions on $G$ generated by the orbit $\{f_g: g\in G\}$ of $f$,
where $f_g(h)=f(gh)$. Let $X$ be the (compact metric) Gelfand 
space corresponding to $\mathcal{A}$. 
The fact that $\mathcal{A}$ is $G$-invariant implies that there is a naturally defined $G$-action on $X$.
Clearly the restriction map $\pi: Z \to X$, where with the above notation $Z$ is
the Gelfand space corresponding to $\mathcal{L}$, is a
homomorphism of the corresponding dynamical systems. We
denote $x_g=\pi(z_g)$, so that $gx_e = x_g$ and 
$\cls \{g x_e : g \in G\} = X$. 

By the construction of the Gelfand space we obtain a natural isomorphism 
of the commutative $C^*$-algebras $\mathcal{A}$ and $C(X)$. 
Let $\hat{f}$ denote the element of $C(X)$
which corresponds to $f$ under this isomorphism.

Clearly the restriction of $\rho$ to $\mathcal{A}$
defines an $m$-stationary probability measure $\mu$ on $X$,
so that the system $(X,\mu,G)$ is an $m$-system.
In fact $\mu = \pi_*(\mu_\rho)$.

Let $A = \{x  \in X: \hat f(x) > 1/2\}$ and consider the set
$N(x_e,A)=\{g \in G: gx_e \in A\}$.  We clearly have
$\{g \in G: f(g) > 1/2\} = \{g \in G: gx_e \in A\}$.
Note that indeed
\begin{align*}
\mu(A) & = \int_X \ch_A \, d\mu =
\int_X \ch_{\{\hat f > 1/2\}} \, d\mu \\
& \ge  \int_X \ch_{\{\hat f=1\}} \, d\mu =   \int_Z \ch_{\{\tilde{f}=1\}} \, d\mu_\rho\\
&\ge \rho(L_{\ep/2}) \ge \rho(L).
\end{align*}

Now assume $g^{-1}_1 A\cap g^{-1}_2 A \cap\cdots 
\cap g^{-1}_k A\ne\emptyset$
and let $x$ be a point in this intersection. Then 
$g_i x \in A$ for $i=1,\dots,k$ and choosing $a \in G$
so that $d(a x_0,x)$ is sufficiently small,
we also have $g_i a x_0 \in A$  for $i=1,\dots,k$.
Thus $f(g_i a ) > 1/2$ and we conclude that
$g_1 a, g_2 a,\dots,g_k a \in L_\ep$.
 \end{proof}

\begin{proof}[Proof of the theorem]
By the correspondence principle, lemma \ref{corres}, we can associate with
$L$ an $m$-system $(X,\mu,G)$ and 
an open
subset $A\subset X$
with $\mu(A)\ge \rho(L) >0$ such that
\begin{equation}\label{corres-eq}
\mu(g^{-1}_1 A\cap g^{-1}_2 A \cap\cdots \cap g^{-1}_k A)>0 
\Longrightarrow g_1 a, g_2 a,\dots,g_k a \in L_\ep,
\end{equation}
for some $a \in G$.

Let
\begin{equation*}
\xymatrix
{
& (X^*,\mu^*) \ar[dl]_{\pi} \ar[dr]^{\sig}  &  \\
(X,\mu) &  & \Pi(\Xcal)=(Y,\la),
}
\end{equation*}
be the canonical standard cover of $(X,\mu)$, given by theorem \ref{2.3}, and
set $A^*=\pi^{-1}(A), B=\sig(A^*)$.

Let
$$
\mu^*=\int_Y \mu_y\times \del_y \, d\la(y)
$$
be the decomposition of $\mu^*$ over $(Y,\la)$.
If we write
$A^*=\bigcup \{A_y\times \{y\}: y\in B\}$ then, by reducing
$B$ if necessary, we can assume that $\mu_y(A_y)\ge \del$
for some $\del >0$.

In the present situation $Y=\PP^1$, the projective line,
and $\sig(\mu^*)=\la$ is Lebesgue measure. Let $y_0\in B$
be a Lebesgue density point of $B$.

{\it Claim}: If $g$ is a parabolic element of $SL(2,\R)$
with fixed point $y_0$ then for every $N$ 
$$
\la(B\cap g B\cap g^{2}B\cap\cdots \cap g^{N}B) >0.
$$

{\it Proof of claim}:
Of course $y_0=gy_0$ is a Lebesgue density point of each
of the sets $g^jB$ and therefore we can choose an interval
$J \subset Y$ such that 
$$
\frac{|J \cap g^jB|}{|J|} \ge (1- \frac{1}{2N}) |J| \quad j=1,\dots,N,
$$ 
whence $\la(J \cap\bigcap_{j=1}^N g^j B) \ge \frac12 |J|$.

Now back to the proof of theorem \ref{sl}.

Szemer\'edi's theorem yields, for every  positive integer $k$ and
$\del> 0$, a positive integer
$N=N(k,\del)$ such that every subset of $\{1,2,\dots,N\}$
of size $\del N$ contains an arithmetic progression of length $k$.

Take $g$ as in the above claim, and such that
$\{g^n: n =1,2,\dots\} \cap Q =\emptyset$.
Denoting $B_0=B\cap g B\cap g^{2}B\cap\cdots \cap g^{N}B$,
we have for $\la$ almost every point $y \in B_0$,
$$
\sum_{i=1}^N \int_X\ch_{A_i}(x)\, d\mu_{g^{i}y}(x) \ge \del N,
$$
where $A_i=A_{g^{i}y}$.
Thus for some $A_{i_j}, 1 \le j \le [\del N]$ the
intersection $\cap_{j=1}^{[\del N]}A_{i_j}\ne\emptyset$.
If we take $N=N(k,\del)$ then we can find
$s$ and $d$ such that
$$
\{s, s+d, s+2d,\dots,s + kd\}\subset
\{i_1,i_2,\dots,i_{[\del N]}\},
$$
and some $x^*\in X^*$ with
$$
g^{s}x^*, g^{s+d}x^*, g^{s+2d}x^*,\dots, g^{s+kd}x^*
$$
all in $A^*$, hence, with $x=\pi(x^*)$,
$$
g^{s}x, g^{s+d}x, g^{s+2d}x,\dots, g^{s+kd}x
$$
all in $A$.

Since there are only finitely many possible $s$ and $d$
we conclude that for some pair $s,d$,
\begin{equation}\label{sd}
\mu(g^{s}A\cap g^{s+d}A \cap\cdots \cap g^{s+kd}A)
 > 0.
\end{equation}
%
%
%
%

In fact 
\begin{gather*}
\mu(g^{s}A\cap g^{s+d}A \cap\cdots \cap g^{s+kd}A)\\
=  \mu^*(g^{s}A^*\cap g^{s+d}A^* \cap\cdots \cap g^{s+kd}A^*)\\
=
\int_{B_0}  \int_X \prod_{j=0}^k\ch_{A_{g^{s+jd}y}}(x)\,
dg^{s+jd}\mu_y(x)\, d\la(y)\\
=\int_{B_0} \int_X \prod_{j=0}^k \ch_{A_{g^{s+jd}y}}(x)\,
d\mu_{g^{s+jd}y}(x)\, d\la(y).
\end{gather*}
Thus, if
$\mu(g^{s}A\cap g^{s+d}A \cap\cdots \cap g^{s+kd}A)=0$ then
also
$$
\int_X \prod_{j=0}^k \ch_{A_{g^{s+jd}y}}(x)\,
d\mu_{g^{s+jd}y}(x)\, d\la(y)=0
$$
for $\la$ a.e. $y\in B_0$, contradicting our assumption on $s$ and $d$.

From (\ref{sd}), by the correspondence principle (\ref{corres-eq}),
we can find $a \in G$ with 
$$
g^{-s}a, g^{-(s + d)}a, \dots, g^{-(s + kd)}a \in L_\ep.
$$
Setting $a_0 = g^{-s}a$ and $h = g^{-d}$ we finally get
$$
a_0, h a_0,\dots, h^k a_0 \in L_\ep.
$$

\end{proof}

\begin{rmk}
Independently of our work T. Meyerovich applies in a recent work \cite{M},
similar ideas in order to obtain multiple and polynomial recurrence
lifting theorems for infinite measure preserving systems.
\end{rmk}

\section{WAP actions are stiff }\label{sec-wap}
A compact topological dynamical system $(X,G)$ is called
{\bf weakly almost periodic}
or {\bf WAP}, if for
every $f\in C(X)$, the set $\{f_g:g\in G\}$ is
relatively compact
in the weak topology on $C(X)$ (where
$f_g \in C(X)$ is defined
by $f_g(x)=f(gx)$.) Ellis and Nerurkar \cite{EN} showed
that $(X,G)$ is weakly almost periodic if and only if
every element $p$ in the enveloping semigroup $E$ of the
system $(X,G)$ is a continuous map.
(Recall that $E=E(X,G)$, the {\bf enveloping semigroup\/}
of the compact topological dynamical system $(X,G)$
is, by definition, the closure of the set of maps
$\{g:X \to X: g\in G\}$ in the compact product space
$X^X$, where the semigroup structure is defined by
composition of maps.)
As shown in \cite{EN} the enveloping semigroup$E=E(X,G)$
of a WAP system contains a
unique minimal left ideal $I$ which is in fact a compact
topological group. Consequently, in a topologically transitive WAP
system there is a unique minimal subset and the action of $G$
on this minimal set is equicontinuous.

A topological dynamical system $(X,G)$ is called
{\bf stiff with respect to $m$} or $m$-{\bf stiff\/} if
every $m$-stationary measure on $X$ is
$G$-invariant, (see \cite{F2}). Our goal in this section is to show that
WAP systems are stiff (theorem \ref{thm5.4} below).

\begin{lem}\label{lem5.1}
Let $(X,G)$ be a WAP dynamical system.
Every element $p\in E$ defines an element $p_*\in E(M(X),G)$
and the map $p\mapsto p_*$ is an isomorphism of $E=E(X,G)$ onto
$E(M(X),G)$. In particular the dynamical system $(M(X),G)$ is also
WAP.
\end{lem}
\begin{proof}
If $g_i\to p$ is a net of elements of $G$ converging to
$p\in E=E(X,G)$,
then by Grothendieck's theorem, for every $f\in C(X)$,
\  $f\circ g_i\to f\circ p$ weakly in $C(X)$.
Therefore, we have for every $\nu\in M(X)$ and $f\in C(X)$:
$$
g_i\nu(f)=\nu(f\circ g_i)\to \nu(f\circ p):=p_*\nu(f).
$$
It is easy to see that $p\mapsto p_*$ is an isomorphism of flows,
whence a semigroup isomorphism. Finally as $G$ is dense in both
enveloping semigroups, it follows that this isomorphism is onto.
\end{proof}

\br

In the sequel we will identify the two enveloping semigroups and will
write $p$ for both $p$ and $p_*$.

\br

We recall the following theorem of R. Azencott (\cite{A},
theorem I.2, page 11).

\begin{thm}\label{contr}
Let $(X,G)$ be a topological dynamical system
with $X$ a compact metric space. Let $\mu$ be a probability
measure on $X$. The following properties are equivalent:
\begin{enumerate}
\item
For every $x\in X$, the measure $\del_x$ is a weak $^*$ limit point
of the set $\{g\mu:g\in G\}$ in $M(X)$.
\item
For every countable dense subset $D$ of $X$ there exists a Borel
subset $A$ of $X$ with $\mu(A)=1$ and with the property that for
every $x\in D$ there exists a sequence $g_n\in G$  such that
$$
\lim g_ny=x\qquad \forall y\in A.
$$
\end{enumerate}
\end{thm}

\br

We call a measure $\mu$ satisfying the
equivalent conditions of theorem \ref{contr} a {\bf contractible} measure,
we then call the dynamical system $(X,\mu,G)$ a
{\bf  contractible system}. Note that a contractible system is
necessarily topologically transitive and moreover every point of
the subset $A$ belongs to the dense $G_\del$ subset $X_{tr}$ of
the transitive points of $X$. Also note that every $m$-proximal
system $(X,\mu,G)$, with $X=\supp(\mu)$, is contractible.

\begin{lem}\label{lem5.3}
Let $(X,G)$ be a WAP system. Let $\mu$ be an $m$-stationary
probability measure on $X$ with $X=\supp(\mu)$
such that the $m$-system $(X,\mu)$
is $m$-proximal. Then $(X,G)$ is the trivial one point system.
\end{lem}
\begin{proof}
Let $X_{tr}$ be the dense, $G$-invariant, $G_\del$ subset of $X$
consisting of all points with dense $G$-orbit.
Let $D$ and $A$ be the subsets
of $X$ given by theorem \ref{contr} (2). Since
$D$ can be any countable dense subset of $X$, we can assume
that $D\subset X_{tr}$. Fix a point $x_0\in D$ and let
$g_n\in G$ satisfy $\lim g_n y=x_0$ for every $y\in A$.
We can assume that the limit $p=\lim g_n$ exists in the enveloping
semigroup $E=E(X,G)$, and then $\lim g_n y=py=x_0$ for every $y\in A$.
Since $p$ is a continuous map and since clearly $A$ is a dense subset
of $X$, it follows that $px=x_0$ for every $x\in X$.
The elements of the left ideal $I=Ep$ are in 1-1
correspondence with the points of $X$ (with $q_x\in I$ defined
by $q_xy=x, \forall y\in X$) and it follows that
$I$ is the unique minimal left ideal in $E$.
It is now clear that $(X,G)$ is a minimal proximal system.
However in a WAP system the group action on the unique minimal
subset is equicontinuous and we conclude that $X$
consists of a single point.
\end{proof}

\br

\begin{thm}\label{thm5.4}
Let $G$ be a locally compact second countable topological group,
$m$ a probability measure on $G$ with the property that the
smallest closed subgroup containing $\supp(m)$ is all of $G$.
Then every WAP dynamical system $(X,G)$ is $m$-stiff.
\end{thm}

\begin{proof}
Let $E=E(X,G)$ be the enveloping semigroup of the WAP system
$(X,G)$.
Let $\mu$ be an $m$-stationary ergodic probability on $X$;
we will show that
$g\mu=\mu$ for every $g\in G$.
As in section 1 we let
$$
\lim_{n\to\infty}\eta_n\mu=\mu_\om,\qquad \ \om\in\Om_0,
$$
be the conditional measures of the $m$-system $\Xcal $, and let
$P^*\in M(M(X))$ be the distribution of the $M(X)$-valued random
variable $\mu(\om)=\mu_\om $. Let now $Z=\supp(P^*)\subset M(X)$.
Clearly $Z$ is a closed $G$-invariant subset of $M(X)$, and by
proposition \ref{1.1} the $m$-dynamical system $(Z,P^*,G)$ is
$m$-proximal. By lemma \ref{lem5.1} the dynamical system $(Z,G)$ is WAP and
therefore, by lemma \ref{lem5.3}, it is the trivial one point system. Since
the barycenter of $P^*$ is $\mu$, we have $P^*=\del_{\mu}$, and it
follows that $P^*$ as well as $\mu$ are $G$-invariant measures.
\end{proof}

\section{The SAT property}
The notion of SAT (strongly approximately transitive) dynamical systems was
introduced by Jaworsky in \cite{J}, where he developed their theory
for discrete groups. For these groups he shows that the stationary measure
on the Poisson boundary is SAT.
It was later used in a slightly stronger version (SAT$^*$)
by Kaimanovich \cite{Ka} in order to study the horosphere foliation on a 
quotient of a CAT$(-1)$ space by a discrete group of isometries $G$, using the 
SAT$^*$ property on the boundary of $G$.

\br

{\bf Definitions}
Let $G$ be a locally compact second countable topological group.
We fix some right Haar measure $m=m_G$ and let $e$ be
the identity element of $G$.


\begin{defns}
\mbox{}
\begin{enumerate}
\item
A {\bf Borel $G$-space\/} is a standard Borel space $(X,\mathcal{X},G)$
with a Borel action $G\times X \to X$.

\item
A {\bf $G$-system\/} is a Borel $G$-space $(X,\mathcal{X},G)$
equipped with a probability measure $\mu$ whose measure class
is preserved by each element of $G$.

\item
We say that
a Borel probability measure $\mu$ on a Borel $G$-space
is {\bf strongly approximately transitive (SAT)\/} if the measure class
of $\mu$ is preserved by each element of $G$ and:
\begin{quote}
For every $A\in \mathcal{X}$ with $\mu(A) > 0$ there is
a sequence $g_n\in G$ such that $\lim_{n\to\infty}\mu(g_n A) =1$.
\end{quote}
When $\mu$ is SAT we will say that the system $(X,\mathcal{X},\mu,G)$ is
SAT.

\item
If $Y$ is a compact metric space, $G$ acts on $Y$ via
a continuous representation of $G$ into $\Homeo(Y)$ and
$\nu$ is a Borel probability measure whose measure class is
preserved by each element of $G$, we will say that
the dynamical system $(Y,\Bcal(Y),\nu,G)$ is
{\bf topological\/} ($\mathcal{B}(Y)$ denotes the Borel field on $Y$).

\item
Let $(X,\mathcal{X},\mu,G)$ be a $G$-system.
A topological $G$-system $(Y,\Bcal(Y),\nu,G)$
is a {\bf topological model\/} for $(X,\mathcal{X},\mu,G)$
if $\supp(\nu)=Y$ and $(X,\mathcal{X},\mu,G)$
and $(Y,\mathcal{B}(Y),\nu,G)$ are isomorphic as
$G$ measure boolean algebras;
i.e. there is an equivariant isomorphism between the corresponding
measure algebras.

\item
Recall that for a topological system $(Y,G)$, we say that a probability measure
$\nu$ on $Y$ is {\bf contractible\/} if for every $y\in Y$ there
exists a sequence $g_n\in G$ such that, in the weak$^*$ topology,
$\lim_{n\to\infty} g_n\nu = \delta_{y}$.

\item
Given a Borel system $(X,\mathcal{X},G)$ we say that a probability
measure $\mu$ on $X$ is {\bf absolutely contractible\/} if for each
topological model $(Y,\nu,G)$ of $(X,\mathcal{X},\mu,G)$
the measure $\nu$ on $Y$ is contractible.
\end{enumerate}
\end{defns}

\br

Our goal is to show that a measure $\mu$ is SAT iff it is
absolutely contractible.

\br

{\bf Contractible topological systems}

We now use a second characterization of contractible measures
(\cite{A}, theorem I.2, page 11).

\begin{thm}\label{4.2}
Let $(Y,G)$ be a topological dynamical system
with $Y$ a compact metric space. Let $\nu$ be a probability
measure on $Y$. The following properties are equivalent:
\begin{enumerate}
\item
The measure $\nu$ is contractible; i.e.
for every $y\in Y$, the measure $\del_y$ is a weak$^*$ limit point
of the set $\{g\nu:g\in G\}$ in $M(Y)$.
\item
The linear operator $P_\nu : C(Y) \to LUC(G)$
defined by
\begin{equation}\label{P}
P_\nu f(g) =\int_Y f(gy)\, d\nu(y),
\end{equation}
is an isometry of the Banach space $C(Y)$ of continuous functions
on $Y$ into the Banach space $LUC(G)$ of bounded left uniformly
continuous functions on $G$ (with $\sup$-norm).
\end{enumerate}
\end{thm}

We note that the operator $P_\nu$ can be extended to the
larger Banach space $L^\infty(Y,\nu)$ using the same formula
\eqref{P}, and since for $f \in L^\infty(Y,\nu)$ and $g, h \in G$
\begin{align*}
\left | P_\nu f (g) - P_\nu f (h) \right | & =
|\langle f, g\nu - h\nu \rangle | \\
& \le \|f\|_\infty \| g\nu  - h\nu\|_{\text {total variation}}\\
& =  \|f\|_\infty \| \nu  - g^{-1} h\nu\|,
\end{align*}
we conclude that $P_\nu(L^\infty(Y,\nu)) \subset LUC(G)$.

\br

{\bf $G$-continuous functions}

Recall the following definition and representation theorem
from \cite{GTW}.

\begin{defn}
Given a $G$-system $(X,\mathcal{X},\mu,G)$, a function $f\in L^\infty(X,\mu)$
is called {\bf $G$-continuous\/} if $f \circ g_n$ converges in norm to
$f$ in $L^\infty(X,\mu)$ whenever $g_n \to e$.
\end{defn}

\begin{thm}\label{2.2}
Let $G$ be a Polish topological group.
A boolean $G$ system admits a topological model if and only if there exists a
sequence of $G$-continuous functions that generates the $\sig$-algebra
(equivalently: separates points).
\end{thm}

\br

We first remark that although in \cite{GTW}
the boolean system is assumed to be measure
preserving the proof, in fact, goes through if one assumes
only that the {\bf measure class\/} is preserved. Next we note that
the condition in the theorem of admitting a
sequence of $G$-continuous functions that generates the $\sig$-algebra,
is always satisfied when the group $G$ is, in addition, locally compact
(see corollary \ref{lc} below).
Thus, in this case, one recovers the classical result that ensures
the existence of a topological model for every measure class
preserving boolean action of a locally compact second countable group.

More importantly, we observe that an immediate corollary of the proof of
theorem 2.2 in \cite{GTW} is the following version of the theorem
(still for a general Polish topological group $G$).

\begin{thm}\label{thm2.3}
Let $(X,\Xcal,\mu,G)$ be a a boolean system which
satisfies the condition of theorem \ref{2.2} and let $f$ be
a function in $L^\infty(X,\mu)$. Then
there exists a topological model $(Y,\nu,G)$ such that
the function $F \in L^\infty(Y,\nu)$ corresponding to $f$
is in $C(Y)$ iff $f$ is $G$-continuous.
\end{thm}

Thus when $G$ is a locally compact second countable group,
theorem \ref{thm2.3} applies for every $G$ system.

\br

\begin{lem}\label{psi}
Let $(X,\Xcal,\mu,G)$ be a $G$-system and
$f\in L^\infty(X,\mu)$. Let $\psi: G \to \R$
be a non-negative continuous function with compact support.
Define $\hat f = f*\psi$ by
$$
\hat f (x) = \int_G f(hx)\psi(h) \, dm_G(h).
$$
The function $\hat f$ is $G$-continuous.
\end{lem}

\begin{proof}

For $g\in G$ we have
\begin{align*}
\| \hat f\circ g - \hat f\|_\infty & =
{\ess\text{-}\sup}_{x\in X}\left |
\int_G f(hgx)\psi(h)\, dm_G(h) - \int_G f(hx)\psi(h)\, dm_G(h)
\right | \\
& =
{\ess\text{-}\sup}_{x\in X}\left |
\int_G f(hx)\psi(hg^{-1})\, dm_G(h) - \int_G f(hx)\psi(h)\, dm_G(h)
\right |\\
& \le
{\ess\text{-}\sup}_{x\in X} \int_G
|f(hx)| |\psi(hg^{-1}) - \psi(h)|\, dm_G(h)\\
& \le
\|f\|_\infty \int_G |\psi(hg^{-1}) - \psi(h)|\, dm_G(h).
\end{align*}
Thus $g\to e$ implies $\| \hat f\circ g - \hat f\|_\infty \to 0$
and $\hat f$ is $G$-continuous.
\end{proof}

\br

We let $\{\psi_n : n=1,2,\dots\}$ be a fixed approximate identity. This
means that there is a decreasing sequence $V_n$ of precompact
neighborhoods of $e$ in $G$ with $\cap_{n=1}^\infty V_n =\{e\}$,
and $\psi_n: G \to \R$ is a sequence
of nonnegative continuous functions
with ${\supp}\, \psi_n \subset V_n$ and $\int_G \psi_n\,dm_G=1$
for $n=1,2,\dots$.

\begin{cor}\label{lc}
The bounded $G$-continuous functions are dense in $L^2(X,\mu)$.
\end{cor}

\begin{proof}
It is easy to check that a sequence $\{\psi_n:n=1,2,\dots\}$ as above
is an approximate identity in $L^2(X,\mu)$; i.e.
$\|f * \psi_n - f\|_2 \to 0$ for every bounded $f \in L^2(X,\mu)$.
Now apply lemma \ref{psi}.
\end{proof}

\br

\begin{prop}\label{e}
Let $(X,\Xcal,\mu,G)$ be a $G$-system and
$0\ne f=\ch_A \in L^\infty(X,\mu)$. Let $\psi_n: G \to \R$
be an approximate identity in $L^2(X,\mu)$ as above. Then
$$
\lim_{n\to\infty}\|f * \psi_n\|_\infty = \|f\|_\infty=1.
$$
\end{prop}

\begin{proof}
With no loss in generality we can assume that $(X,\mu,G)$
is a topological model. By the regularity of the measure $\mu$
we can also assume (by passing to a subset) that $A$ is closed.
Again by regularity of $\mu$, given $\ep>0$ we can choose
an open neighborhood $U$
of $A$ in $X$ such that $\mu(U\setminus A) < \ep$.
Since
$$
\lim_{h\to e} \mu(hA \tri A)=\lim_{h\to e} \|\ch_{hA} - \ch_A\|_1 = 0,
$$
we can choose a neighborhood $V=V^{-1}$ of $e$ in $G$
such that for all $h\in V$
\begin{equation*}
(i)\quad
\|\ch_{hA} - \ch_A\|_1= \mu(hA \tri A) < \ep \quad
{\text{and }} \quad
(ii)\quad  hA \subset U.
\end{equation*}
Let $\psi=\psi_n$ be a member
of the approximate identity which satisfies
$\supp(\psi)\subset V$.

Set
$$
\hat f (x) = f*\psi = \int_G f(hx)\psi(h) \, dm_G(h)
= \int_G f(hx)\, dp(h),
$$
where $dp = \psi \cdot dm_G$, a probability measure
on $G$.
By (ii), if $x\not\in U$ then  $f(hx) = \ch_A(hx)=0$ for every $h\in V$
and it follows that $\hat f(x) =0$.
Thus
\begin{equation}
\int_X \hat f \,d\mu(x) = \int_U \hat f \,d\mu(x).
\end{equation}
By Fubini and the estimation (i),
\begin{equation}
\int_X \hat f \,d\mu(x) = \int_X \mu(hA) \,dp(h) \ge \mu(A) - \ep.
\end{equation}

For $\del>0$ let
$$
D=D_\del = \{x \in U : \hat f(x) < 1 -\del \}.
$$
Then
\begin{align*}
\mu(A) -\ep & \le \int_X \hat f \,d\mu(x) = \int_U \hat f \,d\mu(x)\\
& = \int_D \hat f \,d\mu(x) + \int_{U\setminus D} \hat f \,d\mu(x)\\
& \le (1 - \del)\mu(D) + \mu(U \setminus D)\\
& =  - \del \mu(D) + \mu(U)\\
& \le - \del \mu(D) + \mu(A) + \ep,
\end{align*}
hence $\mu(D) \le \frac{2\ep}{\del}$.
Fixing $\del$ at the outset we choose $\ep$
so that, say, $\frac{2\ep}{\del} \le \frac{1}{2}\mu(A)$,
and then for sufficiently large $n$,
$$
\hat f(x) =  f * \psi_n(x) \ge 1 - \del,
$$
for every $x$ in the set $U \setminus D$ whose measure
$\mu(U \setminus D) \ge \frac{1}{2}\mu(A)>0$.
This completes the proof of the proposition.
\end{proof}

\br

{\bf Sat and absolute contractibility are equivalent}

\begin{thm}\label{sat=az}
Let $(X,\Xcal,\mu,G)$ be a $G$ system, then $\mu$ is SAT
iff it is absolutely contractible.
\end{thm}

\begin{proof}
Let $(Y,\nu,G)$ be a compact model and let $f \in L_\infty(\nu)$
with $\|f\|_\infty =1$ and $\ep>0$ be given. Set $A=\{y \in Y: |f(y)|\ge 1-\ep\}$.
Then $\nu(A)>0$ and we now assume that also 
$\nu(A_+)>0$ where $A_+=\{y \in Y:  f(y) \ge 1-\ep\}$.
By assumption there is a sequence $g_n \in G$ such that
$\lim_{n\to\infty}\nu(g^{-1}_n A_+) =1$.
Hence
\begin{align*}
P_\nu f(g_n) &= \int_{g^{-1}_nA_+} f(g_n y) \,d\nu(y) + 
\int_{g^{-1}_nA^c_+}f(g_n y) \,d\nu(y)\\
& \ge (1-\ep) \nu(g^{-1}_nA_+)  -   \nu(g^{-1}_nA^c_+) \to 1 -\ep.
\end{align*}
Hence $\limsup_{n\to\infty}P_\nu f(g_n)\ge 1 -\ep$.
Similarly when $\nu(A_-)>0$ with $A_-=\{y \in Y:  f(y) \le -1 +\ep\}$
we get $\limsup_{n\to\infty}|P_\nu f(g_n)|\ge 1 -\ep$ .
Thus the $LUC(G)$ norm $\|P_\nu f\| \ge 1 -\ep$
and as this holds for every $\ep$ we get $\|P_\nu f\| = 1$.
Thus $P_\nu : L^\infty(Y,\nu) \to LUC(G)$ is
an isometry. In particular $P_\nu : C(Y) \to LUC(G)$
is an isometry and by theorem \ref{4.2}, $(Y,\nu,G)$ is
contractible.

Conversely, assume that $\mu$ is absolutely contractible
and let $A \in \Xcal$ be a set with positive $\mu$ measure.
Write $f = \ch_A $.  By proposition \ref{e}, given $\ep>0$,
we can choose $\psi=\psi_n: G \to \R$,
a function in the approximate identity, such that
for $\hat f = f*\psi$,
\begin{equation}\label{eq1}
\big | \|\hat f\|_\infty -  1 \big | =
\big | \|\hat f\|_\infty -  \|f\|_\infty \big |
< \ep.
\end{equation}

By lemma \ref{psi}, the function $\hat f$ is $G$-continuous and
by theorem \ref{2.3} there is a topological model
$(Y,\nu,G)$ for $(X,\Xcal,\mu,G)$ in which the
$L^\infty(Y,\nu)$ function corresponding to $\hat f$,
say $F$, is in $C(Y)$. By assumption the measure
$\nu$ on $Y$ is contractible and thus by theorem \ref{4.2},
$P_\nu(F) = P_\mu(\hat f)\in LUC(G)$ satisfies
$$
\|P_\mu(\hat f)\| = \|F\| = \|\hat f\|_\infty.
$$
Let $g\in G$ satisfy
\begin{equation}\label{eq2}
\|P_\mu(\hat f)\| < P_\mu(\hat f)(g) + \ep.
\end{equation}

Now
\begin{align*}
P_\mu(\hat f)(g) & = \int_X \hat f(gx)\, d\mu(x)\\
& = \int_X \int_G  f(hgx)\psi(h)\, dm(h)\, d\mu(x)\\
& = \int_G \psi(h)\left(\int_X  f(hgx)\, d\mu(x)\right)\, dm(h)\\
& = \int_G \psi(h) P_\mu(f)(hg) \, dm(h),
\end{align*}
and, since $\psi \ge 0$ and $\int_G \psi\,dm=1$, it follows that
for some $h\in G$
\begin{equation}\label{eq3}
P_\mu(f)(hg) > P_\mu(\hat f)(g) - \ep.
\end{equation}

Collecting the estimations \eqref{eq1}, \eqref{eq2} and \eqref{eq3} we get
\begin{equation*}
P_\mu(f)(hg) > P_\mu(\hat f)(g) - \ep
> \|P_\mu(\hat f)\| - 2 \ep =
 \|\hat f\|_\infty - 2 \ep
>  1 - 3 \ep.
\end{equation*}
Explicitly
\begin{equation*}
P_\mu(f)(hg) =
\int_X f(hgx)\,d\mu(x) = \mu(hgA)
>  1 - 3 \ep
\end{equation*}
and the proof is complete.
\end{proof}

\br


\end{document}